\documentclass[11pt,twoside,a4paper]{amsart}
\usepackage[inner=2cm,outer=2cm,top=2.5cm,headsep=1cm, footskip=1cm,bottom=3cm]{geometry}

\usepackage{amssymb, amscd, amsmath, enumerate,verbatim}

\usepackage{color}
\usepackage{graphicx}
\usepackage[utf8]{inputenc}
\usepackage[table]{xcolor}

\input xy
\xyoption{all}

\numberwithin{equation}{section}
\newtheorem{theorem}{Theorem}[section]
\newtheorem{lemma}[theorem]{Lemma}
\newtheorem{proposition}[theorem]{Proposition}
\newtheorem{corollary}[theorem]{Corollary}

\theoremstyle{definition}

\newtheorem{remark}[theorem]{Remark}
\newtheorem{remarks}[theorem]{Remarks}
\newtheorem{example}[theorem]{Example}

\newtheorem{definition}[theorem]{Definition}

\newcommand{\jb}{\tiny{$\bullet$}}

\newcommand{\pb}{\ar@{}[dr]|{\mbox{\LARGE{$\lrcorner$}}}} 
\newcommand{\lra}{\longrightarrow}

\newcommand{\grisf}{\cellcolor{black!10}}
\usepackage{caption}

\newcommand{\mk}{\mathbf{k}}

\newcommand{\D}{\mathcal{D}}

\newcommand{\M}{\mathcal{M}}

\newcommand{\Alg}{\mathbf{Alg}}

\newcommand{\Op}{\mathbf{Op}}

\newcommand{\QQ}{\mathbb{Q}}

\newcommand{\Ho}{\mathrm{Ho}}

\newcommand{\Ker}{\mathrm{Ker}\,}

\newcommand{\id}{\mathrm{id}}

\newcommand{\Lie}{\mathcal Lie}
\newcommand{\Com}{\mathcal Com}
\newcommand{\Ass}{\mathcal Ass}
\newcommand{\Ger}{\mathcal Ger}

\newcommand{\alglibre}[2]{#1 \left\langle #2 \right\rangle }
\newcommand{\foa}[1]{P\langle #1 \rangle}

\begin{document}

\title{Sullivan Minimal models of operad algebras}

\author{Joana Cirici}
\address[J. Cirici]{Fachbereich Mathematik und Informatik, Universit\"{a}t M\"{u}nster\\
Einsteinstraße 62\\
48149 Münster\\
Germany }

\author{Agust\'{\i} Roig}
\address[A. Roig]{Dept. Matemàtiques \\ Universitat Polit\`{e}cnica de Catalunya, UPC \\ Diagonal 647, 08028 Barce\-lo\-na, Spain.}

\thanks{Cirici would like to acknowledge financial support from the German Research Foundation (SPP-1786) and partial support from the Spanish Ministry of Economy and Competitiveness (MTM2016-76453-C2-2-P). Roig partially supported by projects MTM2012-38122-C03-01/FEDER and 2014 SGR 634
}


\begin{abstract}We prove the existence of Sullivan minimal models of operad algebras, for a quite wide family of operads in the category of complexes of vector spaces over a field of characteristic zero. Our construction is an adaptation of Sullivan's original \textit{step by step} construction to the setting of operad algebras.
The family of operads that we consider includes all operads concentrated in degree 0 as well as their minimal models. In particular, this gives Sullivan minimal models for algebras over $\Com$, $\Ass$ and $\Lie$, as well as over their minimal models $\Com_\infty$, $\Ass_\infty$ and $\Lie_\infty$. Other interesting operads, such as the operad $\Ger$ encoding Gerstenhaber algebras, also fit in our study.
\end{abstract}

\maketitle
\tableofcontents

\section{Introduction}

The classical construction of Sullivan minimal models of commutative differential
graded algebras over a field $\mk$ of characteristic zero,
is done step by step by a process of ``attaching cells'', called \textit{KS-extensions} (from Koszul-Sullivan) or \textit{Hirsch extensions}.
The data of these KS-extensions is
encoded in a graded vector space together with a linear differential,
whereas the multiplication of the algebra comes for free, thanks to the notion of free algebra.
With this in mind, it is natural to ask whether the cell attachment construction can be
extrapolated to a more general context.
An obvious candidate is the category of $P$-algebras,
where $P$ is an operad in the category of complexes of $\mk$-vector spaces.

While $P$-algebras can behave very badly, in the sense that operations with negative degrees can undo the work of previous steps in a cell attachment procedure,
many interesting operads given in nature (i.e. geometry, topology and physics) behave badly, but in a somewhat \emph{tame} way that we precise here:

Let $P$ be an operad in cochain complexes of $\mk$-vector spaces.
We will always assume that $P$ is connected ($P(1)=\mk$) and that it is either reduced ($P(0)=0$) or unitary ($P(0)=\mk$).
Let $r\geq 0$ be an integer.
We say that $P$ is \textit{$r$-tame} if for all $n\geq 2$, we have that
\[
P(n)^q=0\text{ for all }q\leq (1-n)(1+r) \ .
\]
Note that $r$-tame implies $(r+1)$-tame.
Examples of $0$-tame operads are: $\Ass$, $\Com$ and $\Lie$,
every operad concentrated in degree 0 and the operads $\Ass_\infty$, $\Com_\infty$ and $\Lie_\infty$.
More generally, minimal models of reduced $r$-tame operads are $r$-tame. An
example of $1$-tame operad is $\mathcal{G}er$, the one encoding Gerstenhaber algebras.

In the category of $P$-algebras, there is a notion of free $P$-algebra generated by a graded $\mk$-vector space.
From this notion, we define \textit{KS-extensions of free $P$-algebras} analogously to the rational homotopy setting of $\Com$-algebras.
We say that a $P$-algebra $\M$ is a \textit{Sullivan minimal $P$-algebra}
if it is the colimit of a sequence of KS-extensions
starting from $P(0)$,
ordered by non-decreasing positive degrees.
A \textit{Sullivan minimal model} of a $P$-algebra $A$ is
a Sullivan minimal $P$-algebra $\M$, together with a morphism $f:\M\to A$ of $P$-algebras
whose underlying map of cochain complexes induces an isomorphism in cohomology; i.e., a
quasi-isomorphism of $P$-algebras.
As a warning, let us remark that here and elsewhere along the paper,
$f$ and all algebra morphisms are morphisms in the strict sense, not $\infty$-morphisms.

As in the rational homotopy setting, we require cohomological connectedness for our algebras.
A $P$-algebra $A$ is called \textit{$0$-connected}
if $H^i(A)=0$ for all $i<0$ and the unit map $\eta:P(0)\to A$ induces an isomorphism $P(0)\cong H^0(A)$.
Let $r\geq 0$. Then $A$ is called \textit{$r$-connected} if, in addition, we have that $H^1(A)=\cdots =H^r(A)=0$.
We prove:

\newtheorem*{I1}{\normalfont\bfseries Theorem $\textbf{\ref{existenceminimal}}$}
\begin{I1}
Let $P$ be an $r$-tame operad.
Then every $r$-connected $P$-algebra $A$ has a Sullivan minimal model
$f:\M\to A$ with  $\M^0=P(0)$ and $\M^i=0$ for all $i<r$ with $i\neq 0$.
Furthermore, if $A$ is $(r+1)$-connected and $H^*(A)$ is of finite type, then $\M$ is of finite type.
\end{I1}
Note that in the particular case $P=\Com$ we recover Sullivan's
theorem of minimal models for commutative differential graded algebras over $\mk$.
We also obtain Sullivan minimal models for $0$-connected $P$-algebras,
when $P$ is one of the operads
$\Ass$, $\Lie$,  $\Com_\infty$, $\Ass_\infty$ or $\Lie_\infty$ among others.
Furthermore, the above result gives Sullivan minimal models for $1$-connected
$\Ger$-algebras. All these minimal models are unique:

\newtheorem*{I2}{\normalfont\bfseries Theorem $\textbf{\ref{unicitat}}$}
\begin{I2}
Let $P$ be an $r$-tame operad and $A$ an $r$-connected $P$-algebra.
Let $f:\M\to A$ and $f':\M'\to A$ be two Sullivan minimal models of $A$.
Then there is an isomorphism $g:\M\to \M'$, unique up to homotopy,
such that $f'g\simeq f$.
\end{I2}

\begin{remarks}A few remarks are in order:
\begin{enumerate}[(1)]
 \item  \textit{Relation with existing Sullivan minimal models.}
Sullivan's classical construction of minimal models for commutative differential graded
algebras has been adapted to several other algebraic settings. Examples are
Quillen's models of differential graded Lie algebras \cite{Q2},
the models for chain differential graded  (Lie) algebras
of Baues-Lemaire \cite{BL} and Neisendorfer \cite{Ne},
the theory of Leibniz algebras of \cite{Livernet2}
and, more closely related to our approach, the minimal models
of chain $P$-algebras, where $P$ is a Koszul operad concentrated in degree 0,
developed by Livernet in her PhD Thesis \cite{Livernet1}.
As we show in Section $\ref{SecChain}$, our results are equally valid for cochain and chain algebras,
after minor modifications are taken into account.
In particular, our work generalizes all of the above mentioned studies.
Furthermore, the results of this paper make precise some of the ideas contained in \cite{Sumaster}, where Sullivan
defines triangular $P$-algebras as free $P$-algebras with a partial ordering on their generators
and sketches a theory of triangular resolutions.
\\

\item \textit{Koszul duality theory.}
For Koszul quadratic $P$-algebras, there is a theory of quasi-free resolutions which give minimal models in some situations
(see \cite{LV}, \cite{Milles1}, \cite{Milles2}).
While there is a certain overlap of algebras for which both Koszul duality theory and
our Sullivan algorithm for algebras over tame operads apply, let us mention some notable differences.
First, to know whether an operad is Koszul or not, can prove to be very difficult (see \cite{MSS}, Remark 3.98).
The theory developed in this paper doesn't require operads to be Koszul, not even quadratic. In particular,
there is no restriction on the height of the relations among their generators.
In contrast, we do impose some restrictions on the arity-degree range of the elements of the operad, but
this condition is straightforward to verify.
Second, while Koszul duality theory applies to quadratic algebras satisfying certain conditions,
our algorithm applies to all sufficiently connected $P$-algebras, once the operad $P$ is proven to be tame.
Furthermore, we produce minimal models for both
unitary, $P(0)=\mk$, and non-unitary, $P(0)=0$ algebras,
while Koszul duality theory applies only to the latter case.
Lastly, let us mention that in Koszul duality theory, minimal models are constructed via the cobar resolution of the associated coalgebra,
while, in this paper, we give ``step by step'' minimal models, following Sullivan's classical approach.
This may be useful, for instance, to compute partial minimal models up to a certain degree and extract homotopical information.
\\

\item  \textit{Kadeishvili's models.}
There are many results in the literature about ``minimal models'' for operad algebras in the $\infty$-sense.
Prominently, Kadeishvili \cite{Kadeishvili} defined minimal models of $A_\infty$-algebras as $A_\infty$-algebras
with trivial differential.
Similarly, there is the Homotopy Transfer Theorem for $P_\infty$-algebras (see \cite{LV})
and the theory of minimal models for operad algebras developed in \cite{CL}.
As it is well-known, minimal models \`{a} la Kadeishvili do not correspond to minimal models \`{a} la Sullivan,
the main differences being that for the first ones, morphisms are
$\infty$-morphisms and minimality is a vanishing condition on the differential, while for the later, morphisms are strict and minimality
involves freeness and a certain behavior of the (not-necessarily trivial) differential.
However, a characterizing property is shared by the two approaches: every quasi-isomorphism between minimal algebras is an isomorphism.\\

\item  \textit{Minimal models of operads.}
Every reduced operad $P$ in the category of complexes of $\mk$-vector spaces such that $H(P)(1)=\mk$
has a minimal model (defined as a free operad whose differential is decomposable).
Here we study minimal models of the algebras, and not of the operads themselves.
However, there is a relation between the two problems that we address in this paper.
The idea, is that one can consider the category of algebras above all operads
as a fibred category. We show that minimal objects in this category are given by those objects
that are both minimal on the fiber and the base.
This provides a global invariance of our minimal models.
\end{enumerate}
\end{remarks}

We explain the contents of the paper. In Section $\ref{Secprelim}$ we collect well-known results on operads and operad algebras.
In Section $\ref{Sectionhomotopy}$ we develop the basic homotopy theory of operad algebras.
In Section $\ref{SecMinimals}$ we introduce $r$-tame operads and prove the existence of minimal models for algebras over these operads. We also show that the minimal model of every $r$-tame operad is $r$-tame, and give some examples.
Section $\ref{SecUniqueness}$ deals with the uniqueness of our minimal models.
In Section $\ref{SecVariable}$ we study the fibred category of algebras over all operads and give global minimal models in this case.
Lastly, in Section $\ref{SecChain}$ we explain the case of chain operad algebras (with homological degree) and compute an example of $\Ger_\infty$ minimal model.

\section{Preliminaries}\label{Secprelim}
In this first section, we recall some main constructions for operads and operad algebras in the category of cochain complexes of vector spaces over a field of characteristic 0 and fix notation. For preliminaries on operads, we refer to \cite{MSS}, \cite{LV}, \cite{Fresse}
and \cite{KM}. We refer to \cite{GM}, \cite{FHT} and the original paper of Sullivan \cite{Su} for a review of rational homotopy theory.

Throughout this paper, let $\mk$ denote a field of characteristic 0.

\subsection*{Operads in cochain complexes}
We will consider unital symmetric operads in the
category of unbounded cochain complexes of vector spaces over $\mk$.
Denote by $\Op$ the category of such operads.

Given an operad $P$ in $\Op$ we will denote by
\[\gamma^P_{l;m_1,\dots ,m_l} : P(l)\otimes P(m_1) \otimes \dots \otimes P(m_l)
\longrightarrow P(m_1 + \cdots + m_l)\]
its structure morphisms and by $\eta : \mk \to P(1)$
its unit. These morphisms satisfy equivariance, associativity, and unit axioms (see \cite{MSS}, Definition I.4). Alternatively, we can use the 
equivalent data of partial
composition operations
\[
\circ_i : P(m)\otimes P(n) \longrightarrow P(m+n-1)  \ .
\]
An operad $P$ is called \textit{unitary} if $P(0)=\mk$ is concentrated in degree 0.
It is called \textit{reduced} if $P(0)=0$.
We will say that $P$ is \textit{connected} if $P(1)=\mk$ is concentrated in degree 0.
In this paper we will always consider connected operads that are either unitary or reduced.

\subsection{Operad algebras}
Let $P\in\Op$ be an operad.
Denote by $\Alg_P$ the category of $P$-algebras.
For a $P$-algebra $A$, we will denote by
$\theta_A(l): P(l)\otimes_{\Sigma_l} A^{\otimes l} \to A$
its structure morphisms. These are subject
to natural associativity and unit constraints (see \cite{KM}).

Since every $P$-algebra has an underlying cochain complex, we
have a notion of \textit{quasi-isomorphism} in $\Alg_P$ given by those morphisms
of $P$-algebras whose underlying morphism of cochain complexes induces an isomorphism in cohomology.

We next recall some constructions in the category of $P$-algebras that will be used in the sequel.

\subsection{Functorial properties}(c.f. \cite{LV}, 5.2.14).\label{functorial}
Every morphism of operads $F:P\to Q$ induces a {\it reciprocal image\/} or \emph{restriction of scalars}
functor $F^*:\Alg_Q \to\Alg_P$
defined on objects $B \in \Alg_Q$ by the
compositions $\theta_{F^*B}(l) = \theta_B(l) \circ (F(l) \otimes \id_B^{\otimes l}) : P(l) \otimes_{\Sigma_l} B^{\otimes l} \to B$.
Note that $F^*$ preserves quasi-isomorphisms and surjective morphisms, since the underlying complexes remain unchanged.

\subsection{Tensor product} Let $P,Q\in\Op$ be two operads.
Their pointwise tensor product is the operad $P\otimes Q$ whose arity
$l$ is the cochain complex $P(l) \otimes Q(l)$. Given a $P$-algebra $A$ and a $Q$-algebra $B$,
their tensor product as cochain complexes $A\otimes B$ has a natural structure of $(P\otimes Q)$-algebra.
The operad $\Com$ being the unit of our tensor product of operads, one has $P\otimes \Com = P$ and hence
the tensor product $A\otimes K$ of any $P$-algebra $A$ with a $\Com$-algebra $K$  is always a $P$-algebra. This gives a bifunctor
$
\Alg_P \times \Alg_{\Com} \longrightarrow \Alg_P
$
defined on objects by $(A, K) \mapsto A\otimes K$.

\subsection{Free algebras}(c.f. \cite{LV}, Section 5.2.5)\label{freealg}
Let $P\in\Op$ be an operad and let $V$ be a graded vector space. If we forget the differential of $P$, the \textit{free $P$-algebra generated by $V$} is the $P$-algebra
\[\foa{V} = \bigoplus_{m\geq 0} \left( P(m) \otimes_{\Sigma_m} V^{\otimes m} \right)\]
with the structure maps
$\theta(m):P(m)\otimes_{\Sigma_m}\foa{V}^{\otimes m}\to \foa{V}$
given by the composition of the shuffle isomorphism followed by the structure morphisms $\gamma$ of $P$:
\[
\xymatrix{
P(l)\otimes\left(P(m_1)\otimes_{\Sigma_{m_1}}V^{\otimes m_1}\right)\otimes\cdots\otimes \left(P(m_l)\otimes_{\Sigma_{m_l}}V^{\otimes m_l}\right)\ar[d]^{Sh}_{\cong}\\
P(l)\otimes \left(P(m_1)\otimes\cdots\otimes P(m_l)\right)\otimes_{\Sigma_{m_1}\times\cdots\times \Sigma_{m_l}}V^{\otimes (m_1+\cdots+m_l)}\ar[d]^{\gamma_{l;m_1,\cdots,m_l\otimes 1}}\\
P(m_1+\cdots+m_l)\otimes V^{\otimes (m_1+\cdots+m_l)}
}
\]
By the universal property of the free $P$-algebra (\cite{LV}), for any linear map $f:V\to A$ of degree 0, there exists a unique morphism of $P$-algebras $\foa{V} \to A$ that restricted to $V$ agrees with $f$.

\begin{remark} Note that the formula for the free $P$-algebra generated by $V$ also makes sense if $P$ carries a non-trivial differential. 
Also, it has the universal property of a free object in the category of $P$-dg algebras for graded vector spaces $V$ with zero differentials 
and $\mk$-linear graded maps $f : V \longrightarrow ZA$.
\end{remark}

\subsection{Cone of a morphism} Given a morphism $f:A\to B$ of $P$-algebras, we denote by $C(f)$ the \textit{cone of $f$}.
This is the cochain complex given by $C(f)^n=A^{n+1}\oplus B^n$ with differential $d(a,b)=(-da,-fa+db)$.
The morphism $f$ is a quasi-isomorphism of $P$-algebras if and only if $H^*(C(f))=0$.

\section{Basic homotopy theory of operad algebras}\label{Sectionhomotopy}
Throughout this section we let $P\in\Op$ be a fixed operad in the category of cochain complexes of vector spaces over $\mk$. We first introduce KS-extensions of $P$-algebras and prove that they satisfy the lifting property with respect to surjective quasi-isomorphisms. Then, we give some main properties of homotopies between morphisms of $P$-algebras.

\begin{remark} In general, in order to define extensions one would require the not easy notion of \emph{tensor product} of $P$-algebras (see \cite{Hin01}, \cite{SU04}, \cite{MS06}, \cite{Lod11}, for instance). Fortunately, in our case it suffices to consider tensor products of \emph{free} (non-differential) algebras.
\end{remark}

\begin{definition}\label{defKS}Let $n>0$ be an integer. Let $A= \foa{V} \in \Alg_P$ be free as graded algebra.
A \textit{degree $n$ KS-extension} of $A$ is the free graded $P$-algebra
\[
A\sqcup_d\foa{V'}:=\foa{V\oplus V'} \ ,
\]
where $V'$ is a graded vector space of homogeneous
degree $n$ and
$d:V'\to Z^{n+1}(A)$ a $\mk$-linear map.
The differential on $A\sqcup_d\foa{V'}$ is defined as the only $P$-derivation extending $d$. This derivation squares
to zero because so do its restrictions to $P$, $A$ and $V'$.
\end{definition}

We have the following universal property for KS-extensions:
\begin{lemma}\label{univ_morfisme}
Let $A\sqcup_d\foa{V'}$ be a KS-extension of a free $P$-algebra $A=\foa{V}$, and let
$f:A\to B$ be a morphism of $P$-algebras. A morphism
$f':A\sqcup_d\foa{V'}\to B$ extending $f$ is uniquely determined by a linear map
$\varphi:V'\to B$ of degree 0 satisfying $d\varphi=fd$.
\end{lemma}

\begin{proof}
By the universal property of free algebras we get $f':\foa{V\oplus V'}\to B$.
To prove that it is compatible with the differentials of $A\sqcup_d\foa{V'}$ and $B$,
it suffices to check this on the restriction to $V'$. We have
$f'\circ d|_{V'}=f\circ d=d\circ \varphi=d\circ f'|_{V'}.$
\end{proof}

KS-extensions satisfy the lifting lemma with respect to surjective quasi-isomorphisms:

\begin{lemma}\label{lifting_extensions} Let $i:A \to A\sqcup_d \foa{V}$ be a KS-extension of degree $n$ and
\[
\xymatrix{
{A}   \ar[r]^{f}  \ar[d]_{i}     &     B \ar@{>>}[d]^{w}_{\wr} \\
{A\sqcup_d \foa{V}}  \ar[r]^{g}\ar@{-->}[ur]^{g'}  &     C
}
\]
a solid commutative diagram of $P$-algebra morphisms, where $w$ is a surjective quasi-isomorphism.
Then, there is a $P$-algebra morphism $g'$ making both triangles commute.
\end{lemma}

\begin{proof}
Consider the solid diagram of $\mk$-vector spaces
\[
\xymatrix{
&&Z^n(C(1_B))\ar@{->>}[d]^{1\oplus w}\\
V\ar@{.>}[urr]^\mu\ar[rr]^{\lambda}&&Z^n(C(w))&.
}
\]
where $\lambda = (f\circ d, g_{| V})$.
Since $w$ is a surjective quasi-isomorphism, this is well defined and $1\oplus w$ is surjective.
Therefore there exists a dotted arrow $\mu=(\alpha,\beta)$ making the diagram commute.
It is straightforward to see that the image of the linear map
$
(d,\beta) : V \to A^{n+1} \oplus B^n
$
is included in $Z^n(C(f))$.
According to the universal property of KS-extensions of Lemma $\ref{univ_morfisme}$,
we may obtain $g'$ as the morphism induced by $g|_A$ together with  $\beta:V\to B^n$.
\end{proof}

\begin{definition}
A \textit{Sullivan $P$-algebra} is the colimit $\M=\cup_{i\geq 0}\M[i]$
of a sequence
\[
\M[0]=P(0)\to \M[1] = \foa{V[1]}\to \M[2]=\M[1]\sqcup_d\foa{V[2]}\to\cdots
\]
of KS-extensions of non-negative degrees, starting from $P(0)=\foa{0}$.
\end{definition}

\begin{proposition}\label{strict_lifting}
Let $C$ be a Sullivan $P$-algebra. For every solid diagram of $P$-algebras
\[
\xymatrix{
&A\ar@{->>}[d]^{w}_{\wr}\\
\ar@{.>}^g[ur]C\ar[r]^f&B&
}
\]
in which $w$ is a surjective quasi-isomorphism, there exists $g$ making the diagram commute.
\end{proposition}
\begin{proof}
Assume that $C'\to C=C'\sqcup_d\foa{V}$ is a KS-extension of degree $n$,
and that we have constructed $g':C'\to A$ such that $wg'=f'$, where $f'$ denotes the restriction of $f$ to $f'$.
The existence of $g$ extending $g'$ now follows from Lemma  $\ref{lifting_extensions}$.
\end{proof}

\begin{remark} The above proposition says that Sullivan $P$-algebras 
are cofibrant objects in the Quillen model structure of the category of $P$-algebras of \cite{Hin97}.
In fact, Sullivan $P$-algebras correspond to the \textit{standard cofibrations} of Hinich.
\end{remark}

The following are standard consequences of
Proposition $\ref{strict_lifting}$.
The proofs are straightforward adaptations of the
analogous results in the setting of $\Com$-algebras
(see Section 11.3 of \cite{GM}, see also Section 2.3 of \cite{Cirici} for proofs
in the abstract setting of categories with a functorial path).

Denote by $\mk [t,dt]$ the $\Com$-algebra with a generator $t$ in degree zero,
a generator $dt$ in degree one, and $d(t) = dt$. We have the unit $\iota $ and evaluations $\delta^0$ and $\delta^1$ at $t=0$ and $t=1$ respectively,
which are morphisms of $\Com$-algebras satisfying $\delta^0 \circ \iota=\delta^1 \circ \iota=1$.

\begin{definition}\label{path}
A \textit{functorial path} in the category of $P$-algebras is defined as the functor
 \[-[t,dt]:\Alg_P\longrightarrow \Alg_P\] given on objects
 by $A[t,dt]=A\otimes \mk[t,dt]$ and on morphisms by $f[t,dt]=f\otimes\mk[t,dt]$, together with the natural transformations
\[\xymatrix{A\ar[r]^-{\iota}&A[t,dt] \ar@<.6ex>[r]^{\delta^1} \ar@<-.6ex>[r]_{\delta^0}&A}\,\,\,;\,\,\, \delta^k\circ \iota=1\]
 given by $\delta^k=1\otimes \delta^k:A[t,dt]\to A\otimes\mk=A$ and $\iota=1\otimes\iota:A=A\otimes\mk\to A[t,dt]$.
\end{definition}
Note that the map $\iota$ is a quasi-isomorphism of $P$-algebras while the maps $\delta^0$ and $\delta^1$ are
surjective quasi-isomorphisms of $P$-algebras.

The functorial path gives a natural notion of homotopy between morphisms of $P$-algebras:

\begin{definition}
Let $f,g:A\to B$ be two morphisms of $P$-algebras. An \textit{homotopy from $f$ to $g$} is given by a morphism
of $P$-algebras $h:A\to B[t,dt]$ such that $\delta^0\circ h=f$ and $\delta^1\circ h=g$.
We use the notation $h:f\simeq g$.
\end{definition}

The homotopy relation defined by a functorial path is reflexive and compatible with the composition
(see for example~\cite[ Lemma I.2.3]{KampsPorter}.
Furthermore, the symmetry of $\Com$-algebras $\mk[t,dt]\lra \mk[t,dt]$ given by $t\mapsto 1-t$ makes
the homotopy relation into a symmetric relation. However, the homotopy relation is not transitive in general.
As in the rational homotopy setting of $\Com$-algebras, we have:

\begin{proposition}\label{transitive}
The homotopy relation between morphisms of $P$-algebras is an equivalence relation for those
morphisms whose source is a Sullivan $P$-algebra.
\end{proposition}
\begin{proof}
It only remains to prove transitivity.
Let $C$ be a Sullivan $P$-algebra and consider morphisms $f,f',f'':C\to A$
together with homotopies $h:f\simeq f'$ and $h':f'\simeq f''$.
Consider the pull-back diagram in the category of $P$-algebras
\[
\xymatrix{
A[t,dt,s,ds] \ar@/_/[ddr]_{\delta^0_t} \ar@/^/[drr]^{\delta^1_s} \ar@{.>}[dr]|-{\pi}\\
& \pb \M\ar[d]\ar[r] & A[t,dt] \ar[d]^{\delta^0_t}\\
& A[s,ds] \ar[r]_{\delta^1_s} & A
}
\]
To see that the map $\pi$ is surjective,
note that if $a(s,ds)$ and $b(t,dt)$ are polynomials such that $a(1,0)=b(0,0)$,
representing an element in $\M$, then
\[\pi(a(s,ds)+b(st,dt)-b(0,0))=(a(s,ds), b(t,dt)).\]
It is straightforward to see that all the $P$-algebras in the above diagram are quasi-isomorphic and that $\pi$ is a quasi-isomorphism.
Consider the solid diagram
\[
\xymatrix{
&&A[t,dt,s,ds]\ar@{->>}[d]^{\pi}_{\wr}\\
C\ar@{.>}[rru]^g\ar[rr]_{(h,h')}&&\M&.
}
\]
Then by Proposition $\ref{strict_lifting}$,
there exists a dotted arrow $g$
such that $\pi g=(h,h')$.
Let $h\widetilde{+}h':=\nabla g$,
where $\nabla:A[t,dt,s,ds]\to A[t,dt]$
is the map given by $t,s\mapsto t$.
This gives the desired homotopy
$h\widetilde{+}h':f\simeq f''$.
\end{proof}

Denote by $[A,B]$ the set of homotopy classes of morphisms of $P$-algebras $f:A\to B$.

\begin{proposition}\label{bijecciohomotopies}Let $C$ be a Sullivan $P$-algebra. Any quasi-isomorphism
$w:A\to B$ of $P$-algebras induces a bijection
$w_*:[C,A]\to [C,B]$.
\end{proposition}
\begin{proof}
 We first prove surjectivity. Consider the mapping path of $w$, given by the pull-back
\[
\xymatrix{\pb\ar[d]_{\pi_1}\M(w)\ar[r]^{\pi_2}&B[t,dt]\ar[d]^{\delta^0}\\
A\ar[r]^w&B&.
}
\]
Define maps $p:=\pi_1:\M(w)\to A$, $q:=\delta^1\pi_2:\M(w)\to B$ and $j:=(1,\iota w):A\to \M(w)$.
We obtain a
solid diagram
\[
\xymatrix{
&A\ar@<.6ex>[d]^{j}\ar@/^2pc/[dd]^{w}\\
&\M(w)\ar@<.6ex>[u]^{p}\ar@{->>}[d]^{q}_{\wr}\\
C\ar@{.>}[ur]^{g'}\ar[r]^{f}&B&,
}
\]
where $q$ is a surjective quasi-isomorphism and
$q j=w$.
By Proposition $\ref{strict_lifting}$,
there exists $g'$ such that $q g'=f$. Let $g:=p g'$.
Then we have 
$f=qg'=\delta^1 \pi_2g'$ and $wg=w\pi_1g'=\delta^0\pi_2g'$.
Therefore $[wg]=[f]$ and $w_*$ is surjective.

To prove injectivity, let $f_0,f_1:C\to A$ be such that $h:wf_0\simeq wf_1$.
Consider the pull-back diagram
\[
\xymatrix{
A[t,dt] \ar@/_/[ddr]_{(\delta^0,\delta^1)} \ar@/^/[drr]^{w[t,dt]} \ar@{.>}[dr]|-{\overline{w}}\\
& \pb\M(w,w)\ar[d]\ar[r] & B[t,dt] \ar[d]^{(\delta^0,\delta^1)}\\
& A\times A \ar[r]_{w\times w} & B\times B
}.
\]
One may verify that $\overline{w}$ is a quasi-isomorphism.
Let $H=(f_0,f_1,h)$ and consider the solid diagram
\[\xymatrix{
&&A[t,dt]\ar[d]^{\overline{w}}_{\wr}\\
C\ar@{.>}[urr]^{G}\ar[rr]^-{H}&&\M(w,w)&.
}\]
Since $\overline{w}_*$ is surjective,
there exists a dotted arrow $G$ such that $\overline{w}G\simeq H$.
It follows that
$f_0\simeq \delta^0G\simeq \delta^1 G\simeq f_1$. Then $f_0\simeq f_1$ by Proposition $\ref{transitive}$.
\end{proof}

\section{Sullivan minimal models}\label{SecMinimals}
In this section, we prove the existence of Sullivan
minimal models of $P$-algebras, for a quite wide family of operads in the category of cochain complexes of $\mk$-vector spaces.

We first introduce the notion of $r$-tame operad. For this class of operads, $r$-connected $P$-algebras will have Sullivan minimal models.

\begin{definition}\label{defrtame}

Let $r\geq 0$ be an integer. An operad $P\in \Op$ is called \textit{$r$-tame} if for all $n\geq 2$,
\[
P(n)^q=0\text{ for all } q\leq (1-n)(1+r) \ .
\]
\end{definition}

Note that $r$-tame implies $(r+1)$-tame for all $r\geq 0$.
Below we represent the condition for being an $r$-tame operad, for $r=0$ and $r=1$.
Elements of $r$-tame operads are allowed to be non-zero in the arity-degree
range determined by the blank squares below, except for
the identity $\id \in P(1)=\mk$, and $P(0)\in\{0,\mk\}$ which are denoted by $*$ and live in arity-degree
$(1,0)$ and $(0,0)$ respectively.

\begin{table}[h]
\begin{tabular}{ r }
\tiny{degree}\\
\tiny{$\vdots$}\\
\tiny{$1$}\\
\tiny{$0$}\\
\tiny{$-1$}\\
\tiny{$-2$}\\
\tiny{$-3$}\\
\tiny{$-4$}\\
\tiny{$-5$}\\
\tiny{$-6$}\\
\tiny{$-7$}\\
\\\\
\end{tabular}
\begin{tabular}{|| c | c | c |c | c | c |c | c | c | c}
\grisf &\grisf&&&&&&&&\reflectbox{$\ddots$}\\\hline
\grisf &\grisf&&&&&&&\\\hline 
\tiny{$\ast$}&\tiny{$\ast$}&&&&&&&\\\hline 
\grisf &\grisf &\grisf &&&&&&\\\hline 
\grisf &\grisf &\grisf &\grisf &&&&&\\\hline 
\grisf &\grisf &\grisf &\grisf &\grisf &&&&\\\hline 
\grisf &\grisf &\grisf &\grisf &\grisf &\grisf &&&\\\hline 
\grisf &\grisf &\grisf &\grisf &\grisf &\grisf &\grisf &&\\\hline 
\grisf &\grisf &\grisf &\grisf &\grisf &\grisf &\grisf &\grisf&\\\hline 
\grisf &\grisf &\grisf &\grisf &\grisf &\grisf &\grisf &\grisf&\grisf\\\hline 
\multicolumn{1}{c}{\text{\tiny{$0$}}}&
\multicolumn{1}{c}{\text{\tiny{$1$}}}&
\multicolumn{1}{c}{\text{\tiny{$2$}}}&
\multicolumn{1}{c}{\text{\tiny{$3$}}}&
\multicolumn{1}{c}{\text{\tiny{$4$}}}&
\multicolumn{1}{c}{\text{\tiny{$5$}}}&
\multicolumn{1}{c}{\text{\tiny{$6$}}}&
\multicolumn{1}{c}{\text{\tiny{$7$}}}&
\multicolumn{1}{c}{\text{\tiny{$8$}}}&
\multicolumn{1}{c}{\text{\tiny{$\cdots$}}}
\end{tabular}
\begin{tabular}{ r }
\\\\\\\\\\\\\\\\\\\\\\\\
\tiny{arity}\\\\
\end{tabular}
\,\,\,
\begin{tabular}{ r }
\tiny{degree}\\
\tiny{$\vdots$}\\
\tiny{$1$}\\
\tiny{$0$}\\
\tiny{$-1$}\\
\tiny{$-2$}\\
\tiny{$-3$}\\
\tiny{$-4$}\\
\tiny{$-5$}\\
\tiny{$-6$}\\
\tiny{$-7$}\\
\\\\
\end{tabular}
\begin{tabular}{|| c | c | c |c | c | c |c | c | c | c}
\grisf &\grisf&&&&&&&&\reflectbox{$\ddots$}\\\hline
\grisf &\grisf&&&&&&&\\\hline 
\tiny{$\ast$}&\tiny{$\ast$}&&&&&&&\\\hline 
\grisf &\grisf & &&&&&&\\\hline 
\grisf &\grisf &\grisf & &&&&&\\\hline 
\grisf &\grisf &\grisf & & &&&&\\\hline 
\grisf &\grisf &\grisf &\grisf & &&&&\\\hline 
\grisf &\grisf &\grisf &\grisf & & & &&\\\hline 
\grisf &\grisf &\grisf &\grisf &\grisf & & &&\\\hline 
\grisf &\grisf &\grisf &\grisf &\grisf & & &&\\\hline 
\multicolumn{1}{c}{\text{\tiny{$0$}}}&
\multicolumn{1}{c}{\text{\tiny{$1$}}}&
\multicolumn{1}{c}{\text{\tiny{$2$}}}&
\multicolumn{1}{c}{\text{\tiny{$3$}}}&
\multicolumn{1}{c}{\text{\tiny{$4$}}}&
\multicolumn{1}{c}{\text{\tiny{$5$}}}&
\multicolumn{1}{c}{\text{\tiny{$6$}}}&
\multicolumn{1}{c}{\text{\tiny{$7$}}}&
\multicolumn{1}{c}{\text{\tiny{$8$}}}&
\multicolumn{1}{c}{\text{\tiny{$\cdots$}}}
\end{tabular}
\begin{tabular}{ r }
\\\\\\\\\\\\\\\\\\\\\\\\
\tiny{arity}\\\\
\end{tabular}
\caption*{$0$-tame operads \hspace{4.5cm} $1$-tame operads}
\end{table}

\begin{definition}\label{defSullmin}Let $r\geq 0$ be an integer. A \textit{Sullivan $r$-minimal $P$-algebra}
is the colimit $\M=\cup_{i\geq 0}\M[i]$ of a sequence of KS-extensions starting from $P(0)$, ordered by non-decreasing degrees bigger than $r$:
\[
\M[0]=P(0)\to \M[1] = \foa{V[1]}\to \M[2]=\M[1]\sqcup_d\foa{V[2]}\to\cdots
\]
with $r<\deg(V[n])\leq \deg(V[n+1])$ for all $n\geq 1$. A \emph{Sullivan $r$-minimal model}
for a $P$-algebra $A$ is a Sullivan $r$-minimal $P$-algebra $\M$ together with a quasi-isomorphism $f: \M \to A$.
\end{definition}

As in the rational homotopy setting, to prove the existence of Sullivan minimal models we will restrict to the case when
our $P$-algebras are cohomologically connected (which we will call \emph{connected} for short from now on).

\begin{definition}
A $P$-algebra $A$ is called \textit{$0$-connected}
if $H^i(A)=0$ for all $i<0$ and the unit map $\eta:P(0)\to A$ induces an isomorphism $P(0)\cong H^0(A)$.
Let $r\geq 0$. Then $A$ is called \textit{$r$-connected} if, in addition, $H^1(A)=\cdots =H^r(A)=0$.
\end{definition}

For the construction of Sullivan minimal models
we will use the following two lemmas. The first of these lemmas ensures that free $P$-algebras
generated by positively-graded vector spaces, are positively-graded when $P$ is tame.

\begin{lemma}\label{positives} Let $V=\bigoplus_{i>r}V^i$ be a graded vector space with degrees $>r$.
If $P$ is $r$-tame then $\foa{V}^0=P(0)$ and
$\foa{V}^{k}=0$ for all $k\leq r$ with $k\neq 0$. In particular, $\foa{V}$ is $r$-connected.
\end{lemma}

\begin{proof}
Let $k\in\mathbb{Z}$. The degree $k$-part of $\foa{V}$ may be written as
\[
\foa{V}^k  = P(0)^k\oplus \left(\sum_{i> r}P(1)^{k-i}\otimes_{\Sigma_1} V^i\right)\oplus\left(
\sum \limits_{\substack{n\geq 2,\\ i_1,\cdots,i_n> r}}P(n)^{q_n}\otimes_{\Sigma_n} V^{i_1}\otimes\cdots\otimes V^{i_n}\right)\ ,
\]
where $q_n=k-i_1-\cdots-i_n\leq k-n(1+r)$.
Since $P(0)^k=0$ for all $k\neq 0$ and $P(1)^{k-i}=0$ for all $k\neq i$, it suffices to
see that for all $n\geq 2$ and all $k\leq r$ we have $P(n)^{q_n}=0$.
Since $P$ is $r$-tame, it suffices to prove that $q_n\leq q_n^*:=(1-n)(1+r)$.
Let $n\geq 2$ be fixed and assume that $k\leq r$. Then
\[q_n=k-i_1-\cdots-i_n\leq k-n(1+r)\leq r-n(1+r)=q_n^*-1<q_n^*.\qedhere\]
\end{proof}

The second lemma characterizes the good behavior of $r$-tame operads with respect to KS-extensions and is inspired in Lemma 10.4 of \cite{GM}.

\begin{lemma}\label{comptes}Let
$V=\bigoplus_{r< i\leq p}V^i$ be a graded vector space with $0\leq r< i\leq p$.
Let $V'$ be a graded vector space of homogeneous degree $p$
and let $P$ be an $r$-tame operad.
Then:
\begin{enumerate}
\item [(1)]$\foa{V\oplus V'}^k=\foa{V}^k$ for all $k<p$ and $\foa{V\oplus V'}^{p}=\foa{V}^p\oplus V'$.
\item [(2)]If $r+1<p$ and $V^{r+1}=0$  then $\foa{V\oplus V'}^{p+1}=\foa{V}^{p+1}$.
\end{enumerate}
\end{lemma}
\begin{proof}
For all $k\in\mathbb{Z}$ we may write
\[
{{\foa{V\oplus V'}^k}\over{\foa{V}^k}} =
\left(\sum_{n\geq 1}P(n)^{q_n}\otimes_{\Sigma_n} V'^{\otimes n}\right)\oplus
\left(\sum \limits_{\substack{n\geq 2,\, 1\leq j\leq n-1\\ r< i_1\leq \cdots\leq i_j\leq p}}
P(n)^{q_n'}\otimes_{\Sigma_n} V^{i_1}\otimes\cdots\otimes V^{i_j}\otimes V'^{\otimes (n-j)}\right),
\]
where $q_n=k-pn$ and $q_n'=k-i_1-\cdots-i_j-p(n-j)$.
We first show that
for $n\geq 2$ and $k\leq p$, we have $P(n)^{q_n}=P(n)^{q_n'}=0$.
Since $P$ is $r$-tame, it suffices to see that both $q_n$ and $q_n'$ are smaller or equal than  $q_n^*:=(1-n)(r+1)$.
Since $r<p$, we have
\[
q_n=k-pn\leq p(1-n)\leq (1-n)(1+r)=q_n^* \ .
\]
Note that $q_n'$ attains its maximum when $k=p$, $j=n-1$ and $i_1=\cdots =i_j=r+1$. Then
\[
q_n'\leq p +(1-n)(1+r)-p=q_n^* \ .
\]
This proves that for $k\leq p$ we have
\[
{{\foa{V\oplus V'}^k}\over{\foa{V}^k}}\cong
P(1)^{k-p}\otimes V'.
\]
Now (1) follows from the fact that $P(1)^{k-p}=0$ for all $k\neq p$ and $P(1)^0=\mk$.

Assume that $p>r+1$ and $V^{r+1}=0$. Then in the above formula for
${{\foa{V\oplus V'}^{p+1}}/{\foa{V}^{p+1}}}$ we have:
if $n>1$ then
\[
q_n=p+1-pn=p(1-n)+1\leq (r+2)(1-n)+1=q_n^*+2-n\leq q_n^*
\]
and $q_1=1\neq 0$.

Note that now $q_n'$ attains its maximum when $j=n-1$ and $i_1=\cdots =i_j=r+2$. Then
for all $n\geq 2$ we have
\[
q_n'=p+1-i_1-\cdots-i_j-p(n-j)\leq p+1+(r+2)(1-n)-p=q_n^*+(2-n)\leq q_n^* \ .
\]
Therefore all the contributions vanish and (2) is satisfied.
\end{proof}

\begin{theorem}\label{existenceminimal} Let $P$ be an $r$-tame operad.
Then every $r$-connected $P$-algebra $A$ has a Sullivan $r$-minimal model
$f:\M\to A$ with  $\M^0=P(0)$ and $\M^i=0$ for all $i<r$ with $i\neq 0$.
Furthermore, if $A$ is $(r+1)$-connected and $H^*(A)$ is of finite type, then $\M$ is of finite type.
\end{theorem}

\begin{proof}
We follow the steps of the classical proof of existence of Sullivan minimal models for $\Com$-algebras 
(see \cite[Theorem 10.3]{GM} for the case of simply connected $\Com$-algebras and \cite[Theorem 13.1]{GM} 
or \cite[Theorem V.8.11]{GMa} for the non-simply connected case).

We will construct, inductively over the degree $n\geq 0$,
a sequence of free $P$-algebras $\M[n]$ together with morphisms of $P$-algebras $f_n:\M[n]\to A$ satisfying the following conditions:

\begin{enumerate}
 \item [$(a_n)$] The $P$-algebra $\M[n]$ Sullivan $r$-minimal and is either equal to $\M[n-1]$ or a composition of KS-extensions
 of degree $n$ of $\M[n-1]$. The morphism $f_n$ extends
 $f_{n-1}$.
 \item [$(b_n)$] The map $H^if_n$ is an isomorphism for all $i\leq n$ and a monomorphism for $i=n+1$.
\end{enumerate}

Then the morphism $f:\bigcup_n f_n:\bigcup_n \M[n] \to A$ will be a Sullivan $r$-minimal model for $A$.
Indeed, condition $(a_n)$ implies that $\M$ is Sullivan $r$-minimal and that $\M^n=\M[k]^n$ for all $k\geq n$.
From $(b_{n+1})$ it follows that $H^n(C(f))=H^n(C(f_{n+1}))=0$. Therefore $f$ is a quasi-isomorphism.

Let $\M[0]=P(0)$. The unit map $\eta:P(0)\to A$
gives a morphism of $P$-algebras $f_0:\M[0]\to A$.
For all $0<i\leq r$ we let $\M[i]=\M[0]$ and $f_i=f_0$. Since $A$ is $r$-connected,
conditions $(a_i)$ and $(b_i)$ are satisfied for all $i\leq r$.

Assume inductively that we have a morphism of $P$-algebras $f_{n-1}:\M[n-1]\to A$ satisfying $(a_{n-1})$ and $(b_{n-1})$.
Condition $(b_{n-1})$ is equivalent to the vanishing of $H^i(C(f_{n-1}))$  for all $i<n$.
Let
\[V[n,0]:=H^n(C(f_{n-1}))\] and consider it as a graded vector space of homogeneous degree $n$.
Take a section of the projection $Z^n(C(f_{n-1}))\twoheadrightarrow V[n,0]$ to obtain a linear differential
$d:V[n,0]\to Z^{n+1}\M[n-1]$ and a linear map $\varphi:V[n,0]\to A^n$
such that $d\varphi=f_{n-1}d$. We then let
\[
\M[n,0]:=\M[n-1]\sqcup_d\foa{V[n,0]}
\]
and denote by $f_{n,0}:\M[n,0]\to A$ the extension of $f_{n-1}$ by $\varphi$.

By Lemma $\ref{positives}$, $\M[n,0]$ is an $r$-connected $P$-algebra. Furthermore, by
$(1)$ of Lemma $\ref{comptes}$ we have that $\M[n,0]^{k}=\M[n-1]^k$ for all $k<n$
and $\M[n,0]^{n}=\M[n-1]^n\oplus V[n,0]$.
In particular, we have $H^if_{n,0}=H^if_{n-1}$ for all $i<n$. Hence by induction hypothesis, $H^if_{n,0}$ is an isomorphism for all $i<n$.
We next prove that $H^nf_{n,0}$ is an isomorphism.
Denote by $j_0:\M[n-1]\to \M[n,0]$ the inclusion.
The morphism of cones $(\id , f_{n,0}):C(j_0)\to C(f_{n-1})$ induces an isomorphism in degree $n$ cohomology
\[H^n(\id , f_{n,0}):H^n(C(j_0))\to H^n(C(f_{n-1})).\]
Indeed, since $\M[n,0]^{n}=\M[n-1]^n\oplus V[n,0]$,
every element in $Z^n(C(j_0))$ may be written as $(x,x'+v)$ where $x,x'\in \M[n-1]$
and $v\in V[n,0]$ are such that $dx=0$ and $dx'+dv=x$. 
The map $H^n(\id , f_{n,0})$ is then given by 
\[[(x,x'+v)]\mapsto [(x,f_{n-1}x'+\varphi v)].\]
To prove surjectivity, note that if $[(x,a)]\in H^n(C(f_{n-1}))$, then there exists $v\in V[n,0]$ with $dv=x$ and $\varphi v=a$.
Therefore $[(x,v)]\in H^n(C(j_0))$ maps to $[(x,a)]$.
To prove injectivity, note that every element in $H^n(C(j_0))$ admits a representative of the form $(dv,v)$.
Then the condition $(dv,\varphi v)=D(x,a)=(dx,f_{n-1}x-da)$ implies that $v=0$.

Now, consider the  morphism of long exact sequences in cohomology
\[
\xymatrix{
H^{n-1}(C(j_0))\ar[d]\ar[r]&H^n(\M[n-1])\ar[d]\ar[r]&H^n(\M[n,0])\ar[d]\ar[r]&H^n(C(j_0))\ar[d]\ar[r]&H^{n+1}(\M[n-1])\ar[d]\\
H^{n-1}(C(f_{n-1}))\ar[r]&H^n(\M[n-1])\ar[r]&H^n(A)\ar[r]&H^n(C(f_{n-1}))\ar[r]&H^{n+1}(\M[n-1]).
}
\]
Since $H^{n-1}(C(j_0))=0$ and $H^n(\id , f_{n,0})$ is an isomorphism,
it follows from the five lemma that $H^nf_{n,0}$ is an isomorphism.

To make $H^{n+1}f_{n,0}$ into a monomorphism, let
\[
V[n,1]:=\Ker(H^{n+1}f_{n,0})=H^n(C(f_{n,0}))\text{ and }\M[n,1]=\M[n,0]\sqcup_d\foa{V[n,1]},
\]
where $V[n,1]$ is considered as a vector space of homogeneous degree $n$ and as in the previous step, we take a section
of the projection $Z^{n}(C(f_{n,0}))\twoheadrightarrow V[n,1]$ to define a differential on $V[n,1]$ and a map 
$f_{n,1}:\M[n,1]\to A$.

Denote by $j_1:\M[n,0]\to \M[n,1]$ the inclusion.
Let $[x]\in H^{n+1}(\M[n,0])$. If $[x]\in \Ker(H^{n+1}f_{n,0})$ then we may write $f_{n,0}x=da$ for some $a\in A$.
The pair $[(x,a)]$ gives an element $v\in \M[n,1]$ with $dv=x$. 
This proves that we have an inclusion
\[\Ker(H^{n+1}f_{n,0})\subset \Ker (H^{n+1}j_1).\]
We iterate the above process by letting
\[
V[n,i]=\Ker(H^{n+1}f_{n,i-1})=H^n(C(f_{n,i-1}))\text{ and }\M[n,i]=\M[n,i-1]\sqcup_d\foa{V[n,i]},
\]
until $\Ker(H^{n+1}f_{n,i})=0$. If this never happens, we let $\M[n]:=\bigcup_i \M[n,i]$ and define
$f_n:\M[n]\to A$ by $f_n|_{\M[n,i]}=f_{n,i}$. 
Reasoning as before, we obtain an inclusion
\[\Ker(H^{n+1}f_{n,i})\subset \Ker ( H^{n+1}j_{i+1}),\]
where $j_i:\M[n,i-1]\to \M[n,i]$ denotes the inclusion.
Let $x\in \Ker(H^{n+1}f_{n})$. Then it has a representative $x_i\in H^{n+1}f_{n,i}$ for some $i$.
But then the inclusion $\Ker(H^{n+1}f_{n,i})\subset \Ker ( H^{n+1}j_{i+1})$ implies that 
the image of $x_i$ in $H^{n+1}f_{n,i+1}$ is trivial. Hence $x=0$.
This proves that $\Ker(H^{n+1}f_{n})=0$.
Since $H^{i}f_{n,i}$ is an isomorphism for each $i\leq n$, it follows that $H^{n}f_{n}$ is an
isomorphism. Therefore $(b_n)$ is satisfied. This ends the inductive step.

If $A$ is $(r+1)$-connected, then we can take $\M[r+1]=\M[r]$ and $(a_{r+1})$ and $(b_{r+1})$ are satisfied.
For $n>r+1$, by Lemma $\ref{comptes}$ we have that
$\M[n,0]^{n+1}=\M[n-1]^{n+1}$. This implies that $\Ker(H^{n+1}f_{n,0})=0$
and hence $\M[n]=\M[n,0]=\M[n-1]\sqcup_d\foa{V[n,0]}$. If $H^*(A)$ has finite type, then
$V[n,0]$ is finite dimensional and $\M[n]$ has finite type.
\end{proof}

Let us review a few examples where Theorem $\ref{existenceminimal}$ applies.

The operads $\Ass$, $\Com$ and $\Lie$ encoding differential graded associative, commutative and
Lie algebras respectively are generated by operations in arity-degree $(2,0)$.
Therefore they are concentrated in degree 0. We have:
\begin{table}[h]
\begin{tabular}{ r }
\tiny{degree}\\
\tiny{$\vdots$}\\
\tiny{$1$}\\
\tiny{$0$}\\
\tiny{$-1$}\\
\tiny{$-2$}\\
\tiny{$-3$}\\
\tiny{$-4$}\\
\tiny{$-5$}\\
\tiny{$-6$}\\
\tiny{$-7$}\\
\\\\
\end{tabular}
\begin{tabular}{|| c | c | c |c | c | c |c | c | c | c}
\grisf &\grisf&&&&&&&&\\\hline
\grisf &\grisf&&&&&&&\\\hline 
\tiny{$\ast$}&\tiny{$\ast$}&\jb&\jb&\jb&\jb&\jb&\jb&\jb&\tiny{$\cdots$}\\\hline 
\grisf &\grisf &\grisf &&&&&&\\\hline 
\grisf &\grisf &\grisf &\grisf &&&&&\\\hline 
\grisf &\grisf &\grisf &\grisf &\grisf &&&&\\\hline 
\grisf &\grisf &\grisf &\grisf &\grisf &\grisf &&&\\\hline 
\grisf &\grisf &\grisf &\grisf &\grisf &\grisf &\grisf &&\\\hline 
\grisf &\grisf &\grisf &\grisf &\grisf &\grisf &\grisf &\grisf&\\\hline 
\grisf &\grisf &\grisf &\grisf &\grisf &\grisf &\grisf &\grisf&\grisf\\\hline 
\multicolumn{1}{c}{\text{\tiny{$0$}}}&
\multicolumn{1}{c}{\text{\tiny{$1$}}}&
\multicolumn{1}{c}{\text{\tiny{$2$}}}&
\multicolumn{1}{c}{\text{\tiny{$3$}}}&
\multicolumn{1}{c}{\text{\tiny{$4$}}}&
\multicolumn{1}{c}{\text{\tiny{$5$}}}&
\multicolumn{1}{c}{\text{\tiny{$6$}}}&
\multicolumn{1}{c}{\text{\tiny{$7$}}}&
\multicolumn{1}{c}{\text{\tiny{$8$}}}&
\multicolumn{1}{c}{\text{\tiny{$\cdots$}}}
\end{tabular}
\begin{tabular}{ r }
\\\\\\\\\\\\\\\\\\\\\\\\
\tiny{arity}\\\\
\end{tabular}
\caption*{The operads $\Ass$, $\Com$ and $\Lie$ are $0$-tame}
\end{table}

The above operads have minimal models, encoding the infinity-versions of their algebras.
These are depicted in the following table.

\begin{table}[h]
\begin{tabular}{ r }
\tiny{degree}\\
\tiny{$\vdots$}\\
\tiny{$1$}\\
\tiny{$0$}\\
\tiny{$-1$}\\
\tiny{$-2$}\\
\tiny{$-3$}\\
\tiny{$-4$}\\
\tiny{$-5$}\\
\tiny{$-6$}\\
\tiny{$-7$}\\
\\\\
\end{tabular}
\begin{tabular}{|| c | c | c |c | c | c |c | c | c | c}
\grisf &\grisf&&&&&&&&\\\hline
\grisf &\grisf&&&&&&&\\\hline 
\tiny{$\ast$}&\tiny{$\ast$}&\jb&\jb&\jb&\jb&\jb&\jb&\jb&\tiny{$\cdots$}\\\hline 
\grisf &\grisf &\grisf &\jb&\jb&\jb&\jb&\jb&\jb&\tiny{$\cdots$}\\\hline 
\grisf &\grisf &\grisf &\grisf &\jb&\jb&\jb&\jb&\jb&\tiny{$\cdots$}\\\hline 
\grisf &\grisf &\grisf &\grisf &\grisf &\jb&\jb&\jb&\jb&\tiny{$\cdots$}\\\hline 
\grisf &\grisf &\grisf &\grisf &\grisf &\grisf &\jb&\jb&\jb&\tiny{$\cdots$}\\\hline 
\grisf &\grisf &\grisf &\grisf &\grisf &\grisf &\grisf &\jb&\jb&\tiny{$\cdots$}\\\hline 
\grisf &\grisf &\grisf &\grisf &\grisf &\grisf &\grisf &\grisf&\jb&\tiny{$\cdots$}\\\hline 
\grisf &\grisf &\grisf &\grisf &\grisf &\grisf &\grisf &\grisf&\grisf&\tiny{$\cdots$}\\\hline 
\multicolumn{1}{c}{\text{\tiny{$0$}}}&
\multicolumn{1}{c}{\text{\tiny{$1$}}}&
\multicolumn{1}{c}{\text{\tiny{$2$}}}&
\multicolumn{1}{c}{\text{\tiny{$3$}}}&
\multicolumn{1}{c}{\text{\tiny{$4$}}}&
\multicolumn{1}{c}{\text{\tiny{$5$}}}&
\multicolumn{1}{c}{\text{\tiny{$6$}}}&
\multicolumn{1}{c}{\text{\tiny{$7$}}}&
\multicolumn{1}{c}{\text{\tiny{$8$}}}&
\multicolumn{1}{c}{\text{\tiny{$\cdots$}}}
\end{tabular}
\begin{tabular}{ r }
\\\\\\\\\\\\\\\\\\\\\\\\
\tiny{arity}\\\\
\end{tabular}
\caption*{The operads $\Ass_\infty$, $\Com_\infty$ and $\Lie_\infty$ are 0-tame}
\end{table}

\begin{corollary}
Let $P$ be one of the operads $\Ass$, $\Com$, $\Lie$, $\Com_\infty$ , $\Ass_\infty$ or $\Lie_\infty$.
Then every $0$-connected $P$-algebra has a Sullivan minimal model. Also,
every $1$-connected $P$-algebra with finite type cohomology has a Sullivan minimal model of finite type.
\end{corollary}

More generally, every reduced operad $P$ such that $H(P)(1)=\mk$ has a minimal model (see Theorem 3.125 in \cite{MSS}).
We next prove that minimal models of reduced $r$-tame operads are $r$-tame. We first introduce some notation.

\begin{definition}Let $P\in\Op$.
Given $w\in P(n)^q$, we will denote by $|w|:=(n,q)$ its \textit{arity-degree}. We will say that $w$ is \textit{$r$-tame} if
$q> (r+1)(1-n)$. Note that $P$ is $r$-tame if and only if all its non-trivial elements of arity $\geq 2$ are $r$-tame.
\end{definition}

\begin{lemma}Every free operad $P\in\Op$ generated by $r$-tame elements is $r$-tame.
\end{lemma}
\begin{proof}
It suffices to show that if $w,w'\in P$ are $r$-tame, then their partial compositions $w\circ_i w'$
are also $r$-tame. Let $|w|=(n,q)$ and $|w'|=(n',q')$. Then
\[q+q'\geq(r+1)(1-n)+1+(r+1)(1-n')+1=(r+1)(1-(n+n'-1))+2.\]
Since $|w\circ_i w'|=(n+n'-1,q+q')$, this implies that $w\circ_i w'$ is $r$-tame.
\end{proof}

\begin{proposition}\label{minimaltameistame}
Let $P\in\Op$ be a reduced $r$-tame operad. Then its minimal model is $r$-tame.
\end{proposition}

\begin{proof}
From the construction of minimal models
of Theorem 3.125 in \cite{MSS}, we easily deduce that for any
reduced operad $P$ with $H(P)(1)=\mk$,
there
is a minimal model $M\to P$ where $M=\cup_{n\geq 2}M_n$ is constructed inductively over the arity $n$
and satisfies:
\begin{enumerate}[(i)]
\item $M_2$ is freely generated by elements of $P(2)$, with $M_2(0)=0$ and $M_2(1)=\mk$.
 \item For $n>2$, $M_n$ is obtained as a free extension of $M_{n-1}$ by
subspaces $A(n,q)\subseteq P(n)^q$ in arity-degree
$(n,q)$ and subspaces $B(n,q)\subseteq  M_{n-1}(n)^q$ in arity-degree $(n,q-1)$.
\end{enumerate}
For our purposes, it is not necessary to know
neither which elements we are adding nor what are their differentials.
We only need to keep track of their possible arities and degrees.

If $P$ is $r$-tame, then $M_2$ is clearly $r$-tame.
Assume inductively that $M_i$ is $r$-tame for all $i<n$.
Property $(ii)$ tells us that $M_n$ is obtained as a free extension of $M_{n-1}$ by
subspaces $A(n,q)\subseteq P(n)^q$ in arity-degree
$(n,q)$ and subspaces $B(n,q)\subseteq  M_{n-1}(n)^q$ in arity-degree $(n,q-1)$.
Elements in $A(n,q)$ are clearly $r$-tame.
Let $w\in B(n,q)$. Since $M_{n-1}$ is generated by elements in arity $<n$,
we may write $w$ as a sum of partial compositions  of the form $ w'\circ_i w''$ where
$w'$ and $w''$ are $r$-tame elements of $M_{n-1}$.
Let $|w'|=(n',q')$ and $|w''|=(n'',q'')$. Then
$|w'\circ_i w''|=(n'+n''-1,q'+q'')=(n,q)$
We get:
\[q=q'+q''\geq (r+1)(1-n')+1+(r+1)(1-n'')+1=(r+1)(1-(n'+n''-1))+2=(r+1)(1-n)+2.\]
This gives $q-1>(r+1)(1-n)$, which is precisely the condition for $w$ to be $r$-tame.
This proves that $M_n$ is $r$-tame and hence $M$ is also $r$-tame.
\end{proof}

\begin{corollary}
Let $P\in\Op$ be a reduced $r$-tame operad and let $P_\infty\to P$ be a minimal model of $P$.
Then every $r$-connected $P_\infty$-algebra has a Sullivan $r$-minimal model. Also,
every $(r+1)$-connected $P_\infty$-algebra with finite type cohomology has a Sullivan $r$-minimal model of finite type.
\end{corollary}

An example of $1$-tame operad is given by the operad encoding \textit{Gerstenhaber algebras}:
these are graded-commutative algebras with a Lie bracket of degree $-1$ satisfying the Poisson identity.
The ordinary multiplication has arity-degree $(2,0)$, while the Lie bracket has arity-degree $(2,-1)$.
We have:
\begin{table}[h]
\begin{tabular}{ r }
\tiny{degree}\\
\tiny{$\vdots$}\\
\tiny{$1$}\\
\tiny{$0$}\\
\tiny{$-1$}\\
\tiny{$-2$}\\
\tiny{$-3$}\\
\tiny{$-4$}\\
\tiny{$-5$}\\
\tiny{$-6$}\\
\tiny{$-7$}\\
\\\\
\end{tabular}
\begin{tabular}{|| c | c | c |c | c | c |c | c | c | c}
\grisf &\grisf&&&&&&&&\\\hline
\grisf &\grisf&&&&&&&\\\hline 
\tiny{$\ast$}&\tiny{$\ast$}&\jb&\jb&\jb&\jb&\jb&\jb&\jb&\tiny{$\cdots$}\\\hline 
\grisf &\grisf &\jb &\jb&\jb&\jb&\jb&\jb&\jb&\tiny{$\cdots$}\\\hline 
\grisf &\grisf &\grisf & \jb&\jb&\jb&\jb&\jb&\jb&\tiny{$\cdots$}\\\hline 
\grisf &\grisf &\grisf & & \jb&\jb&\jb&\jb&\jb&\tiny{$\cdots$}\\\hline 
\grisf &\grisf &\grisf &\grisf & &\jb&\jb&\jb&\jb&\tiny{$\cdots$}\\\hline 
\grisf &\grisf &\grisf &\grisf & & &\jb &\jb&\jb&\tiny{$\cdots$}\\\hline 
\grisf &\grisf &\grisf &\grisf &\grisf & & &\jb&\jb&\tiny{$\cdots$}\\\hline 
\grisf &\grisf &\grisf &\grisf &\grisf & & &&\jb&\tiny{$\cdots$}\\\hline 
\multicolumn{1}{c}{\text{\tiny{$0$}}}&
\multicolumn{1}{c}{\text{\tiny{$1$}}}&
\multicolumn{1}{c}{\text{\tiny{$2$}}}&
\multicolumn{1}{c}{\text{\tiny{$3$}}}&
\multicolumn{1}{c}{\text{\tiny{$4$}}}&
\multicolumn{1}{c}{\text{\tiny{$5$}}}&
\multicolumn{1}{c}{\text{\tiny{$6$}}}&
\multicolumn{1}{c}{\text{\tiny{$7$}}}&
\multicolumn{1}{c}{\text{\tiny{$8$}}}&
\multicolumn{1}{c}{\text{\tiny{$\cdots$}}}
\end{tabular}
\begin{tabular}{ r }
\\\\\\\\\\\\\\\\\\\\\\\\
\tiny{arity}\\\\
\end{tabular}
\caption*{The Gerstenhaber operad $\Ger$ is 1-tame}
\end{table}

\begin{corollary}
Every $1$-connected $\Ger$-algebra (resp. $\Ger_\infty$-algebra) has a Sullivan minimal model and
every $2$-connected $\Ger$-algebra (resp. $\Ger_\infty$-algebra) with finite type cohomology has a Sullivan minimal model of finite type.
\end{corollary}

\begin{remark} In the last section of this paper, we will study chain $P$-algebras for operads of chain complexes; that is, both with positive homological degrees and differential of degree $-1$. In this setting, the generator corresponding to the Lie bracket in the Gerstenhaber operad has arity-degree $(2,1)$. In particular,
$\Ger$ is a $0$-tame operad and the restriction to $1$-connected algebras is no longer necessary.
\end{remark}

\section{Uniqueness of the minimal model}\label{SecUniqueness}

In this section we prove the uniqueness of Sullivan minimal models. The proof is parallel to that in the setting of $\Com$-algebras. As in the previous section, the key ingredient is Lemma $\ref{comptes}$.

\begin{lemma}\label{seccio}Let $P$ be an $r$-tame operad and let
$f:A\to \M$ be a quasi-isomorphism of $r$-connected $P$-algebras, with $\M$ a Sullivan $r$-minimal $P$-algebra.
Then there exists a morphism of $P$-algebras $g:\M\to A$ such that $fg=\id_{\M}$.
\end{lemma}

\begin{proof}
We rewrite the proof of G\'{o}mez-Tato (see Lemma 4.4 of \cite{GoTa})
for $\Com$-algebras in the $P$-algebra setting (see also Theorem 14.11 of \cite{FHT} and
\cite{Roig3}, \cite{Roig4} for a categorical version).

By definition, we may write $\M=\M'\sqcup_d\foa{V}$
where $\M'$ is a free $P$-algebra generated by a graded vector
space $V'$ of degrees $r<i\leq p$ and $V$ is a graded vector space
of homogeneous degree $p$, with $p>0$.
Assume inductively that we have a morphism of $P$-algebras $g':\M'\to A$ such that $fg'=1_{\M'}$.
Then $g'$ is injective.
The morphism $f$ induces a morphism of cochain complexes (not of $P$-algebras!)
\[\overline{f}:A/g'(\M')\to {{\M}/{\M'}}\]
which is a quasi-isomorphism.
By Lemma $\ref{comptes}$ we have that $(\M/\M')^{p-1}=0$ and that $(\M/\M')^p= V$.
This gives  a surjection at the level of cocycles
\[\pi:Z^p(A/g'(\M'))\twoheadrightarrow Z^p({{\M}/{\M'}})=V.\]
We obtain a linear map $\varphi:V\to A^p$ such that $f\varphi=1_V$ to a morphism $g:\M\to A$ by taking sections
of the projections $A\to A/g'(\M')$ and $\pi$ and considering the composition
\[V=(\M/\M')^p\to Z^p(A/g'(\M')\hookrightarrow (A/g'(\M'))^p\to A^p.\]
For a proof of this last fact taking elements and checking that everything works fine see the
proof of Theorem 14.11 in \cite{FHT}.
By Lemma $\ref{univ_morfisme}$, the map $\varphi$ extends $f'$ to a morphism $f:\M\to A$.
\end{proof}

As a classical consequence of Lemma $\ref{seccio}$ we have:

\begin{lemma}\label{quisesiso}
Let $P$ be an $r$-tame operad and let
$f:\M\to \M'$ be a quasi-isomorphism of Sullivan $r$-minimal $P$-algebras.
Then $f$ is an isomorphism.
\end{lemma}

\begin{proof}
By Lemma $\ref{seccio}$ we have a morphism $g:\M'\to \M$ such that $fg=\id_{\M'}$.
By the two out of three property, $g$ is also a quasi-isomorphism.
Again, by Lemma $\ref{seccio}$ we have a morphism $g':\M\to \M$ such that $gg'=\id_{\M}$.
Therefore $g$ is both injective and surjective and hence an isomorphism and we have $f=g^{-1}$.
\end{proof}

The main result of this section is the following:

\begin{theorem}\label{unicitat}Let $P$ be an $r$-tame operad and let $A$ be an $r$-connected $P$-algebra.
Let $f:\M\to A$ and $f':\M'\to A$ be two Sullivan $r$-minimal models of $A$.
Then there is an isomorphism $g:\M\to \M'$, unique up to homotopy,
such that $f'g\simeq f$.
\end{theorem}

\begin{proof}
By Proposition $\ref{bijecciohomotopies}$ we obtain $g$, uniquely defined up to homotopy, such that $f'g\simeq f$.
By Lemma $\ref{quisesiso}$, $g$ is an isomorphism.
\end{proof}

Let $\Ho(\Alg_P^r)$ denote the localized category of $r$-connected $P$-algebras with respect to the class of quasi-isomorphisms.
Denote by $\mathbf{SMin}_P^r/\!\simeq $ the category of $r$-connected Sullivan minimal $P$-algebras,
whose morphisms are homotopy classes of morphisms of $P$-algebras. We have:

\begin{corollary}\label{CEAlgsP}
Let $P$ be an $r$-tame operad. The category $\Alg_P^r$ of $r$-connected $P$-algebras is a Sullivan category in the sense of \cite{GNPR10}.
In particular, the inclusion of minimal algebras induces an equivalence of categories
\[
\mathbf{SMin}_P^r/\!\simeq
\stackrel{\sim}{\lra}
\Ho(\Alg_P^r) \ .
\]
\end{corollary}

\begin{proof} For every $P$-algebra $A$, the choice of a minimal model $\M$ gives a well-defined functor $A \mapsto \M$ between the homotopy categories,
which is the quasi-inverse of the functor induced by the inclusion.
\end{proof}

\section{Algebras over variable operads}\label{SecVariable}
Let $P$ be an operad and $A$ a $P$-algebra. Given two minimal models $F:P_\infty\to P$ and $F':P_\infty'\to P$,
we may consider the reciprocal images $F^*(A)$ and $F'^*(A)$ of $A$ in the categories
of $P_\infty$-algebras and $P_\infty'$-algebras respectively.
In this section, we compare the minimal models of these reciprocal images.
This problem is better understood in the fibred
category of \textit{algebras over all operads}, which we next introduce.

\begin{definition}
Denote by $\Alg$ the category whose objects are pairs $(P,A)$ with $P\in\Op$ and $A\in\Alg_P$ and whose morphisms
$(F,f):(P,A)\to (Q,B)$ are given by a morphism $F:P\to Q$ of operads, together with a morphism
$f:A\to F^*(B)$ of $P$-algebras.
The composition of morphisms $(F,f):(P,A)\to (Q,B)$ and $(G,g):(Q,B)\to (R,C)$ is defined by
$(G,g)\circ (F,f):=(G\circ F, F^*(g)\circ f).$
Objects in $\Alg$ will be called \textit{algebras (over variable operads)}.
\end{definition}

Following the main theorem of \cite{Roig1} and taking into account the remarks of \cite{Stan12},
one can produce a Quillen model category structure on $\Alg$,
from the ones on $\Op$ and $\Alg_P$  (see \cite{BM03}, \cite{Hin97}).
However, since here we are only interested in minimal models,
we don't need the whole power of a Quillen model structure.
As we have seen, in order to talk about and prove existence of minimal models
it suffices to consider weak equivalences (quasi-isomorphisms).
If, on top, we want to study uniqueness, we also need a notion of homotopy.

\begin{definition}
A morphism $(F,f):(P,A)\to (Q,B)$ in $\Alg$ is said to be a \textit{quasi-isomorphism}
if  $F: P \to Q$ is a quasi-isomorphism of operads and
$f:A \to F^*(B)$ is a quasi-isomorphism of $P$-algebras.
\end{definition}

Using the notion of principal extension of an operad (\cite{MSS}, definition 3.138), we define Sullivan operads
as done with operad algebras:

\begin{definition}
A \textit{Sullivan operad} is the colimit of a sequence of
principal extensions  of arities $> 1$, starting from $0$.
\end{definition}

\begin{definition}
We will say that a pair $(R,C)\in\Alg$ is a \textit{Sullivan algebra} if
$R$ is a Sullivan operad and $C$ is a Sullivan $R$-algebra.
\end{definition}

\begin{proposition}\label{liftingvariables}
Let $(R,C)$ be a Sullivan algebra. Then for every solid diagram in $\Alg$
\[
\xymatrix@R=1pc@C=1pc{
&&(P,A)\ar@{->>}[dd]^{(W,w)}\\\\
\ar@{.>}[uurr]^{(G,g)}(R,C)\ar[rr]^-{(F,f)}&&(Q,B)&
}
\]
where $(W,w)$ is a surjective quasi-isomorphism, there exists $(G,g)$ making the diagram commute.
\end{proposition}

\begin{proof}
Since $R$ is a Sullivan operad, by Lemma 3.139 of \cite{MSS} there exists a morphism
$G:R\to P$ such that $W\circ G=F$.
Consider the solid diagram of $P$-algebras
\[
\xymatrix{
&G^*(A)\ar@{->>}[d]^{G^*(w)}\\
\ar@{.>}[ur]^{g}C\ar[r]^-{f}&F^*(B)&
}
\]
Note that since $G^*W^*=F^*$, this is well-defined.
Since $C$ is a Sullivan $R$-algebra and $G^*(w)$ is a surjective quasi-isomorphism,
by Proposition $\ref{strict_lifting}$
there is a morphism $g$ making the diagram commute.
\end{proof}

\begin{definition}\label{pathAlg}
A \textit{functorial path} in the category $\Alg$ is defined as the functor
\[
-[t,dt]:\Alg\longrightarrow \Alg
\]
given on objects by $(P,A)[t,dt] = (P[t,dt],A[t,dt])$ and on morphisms by $(F,f)[t,dt]=(F[t,dt],f[t,dt])$, together with the natural transformations

\[\xymatrix{
(P,A)\ar[r]^-{(I,\iota)} & {(P[t,dt],A[t,dt])} \ar@<.6ex>[r]^-{(\Delta^1,\delta^1)} \ar@<-.6ex>[r]_-{(\Delta^0,\delta^0)}  &   {(P,A)}
}
\,\,\,;\,\,\, (\Delta^k, \delta^k)\circ (I,\iota)= \id_{(P,A)} \ .
\]

\end{definition}

Note that if $F:P\to Q$ is a morphism of operads then
$F[t,dt]^*=F^*[t,dt]$.

The path gives a natural notion of homotopy between morphisms in $\Alg$.
As in Section $\ref{Sectionhomotopy}$, the following are classical
consequences of Proposition $\ref{liftingvariables}$.

\begin{proposition}
The homotopy relation between morphisms in $\Alg$ is an equivalence relation for those
morphisms whose source is a Sullivan algebra.
\end{proposition}
\begin{proof}
The proof follows verbatim the proof of Proposition $\ref{transitive}$, using Proposition $\ref{liftingvariables}$.
\end{proof}

Denote by $[(P,A),(Q,B)]$ the set of homotopy classes of morphisms of algebras from $(P,A)$ to $(Q,B)$.

\begin{proposition}\label{bijeccioalgs}
Let $(R,C)$ be a Sullivan algebra.
Any quasi-isomorphism
$(W,w):(P,A)\to (Q,B)$ induces a bijection
$(W,w)_*:[(R,C),(P,A)]\to [(R,C),(Q,B)]$.
\end{proposition}
\begin{proof}
The proof follows verbatim the proof of Proposition $\ref{bijecciohomotopies}$.
\end{proof}

We now study the existence and uniqueness of minimal models in $\Alg$.

\begin{definition}
We will say that $(P_\infty,\M)\in\Alg$ is a \textit{Sullivan $r$-minimal algebra} if
$P_\infty$ is a minimal operad which is $r$-tame and $A$ is a Sullivan $r$-minimal $P_\infty$-algebra. A \emph{Sullivan $r$-minimal model} for a pair $(P,A)$ is a Sullivan $r$-minimal algebra $(P_\infty, \M)$ together with a quasi-isomorphism $(P_\infty, \M) \to (P,A)$.
\end{definition}

\begin{theorem}Let $P$ be a reduced $r$-tame operad and let $A$ be an $r$-connected $P$-algebra.
Then $(P,A)$ has a Sullivan $r$-minimal model.
\end{theorem}

\begin{proof}
By Theorem 3.125 of \cite{MSS}, every reduced operad $P\in\Op$ with $H(P)(1)=\mk$ has a minimal model $F:P_\infty\to P$.
Since $P(1)=\mk$ this hypothesis is clearly satisfied.
Furthermore, $P_\infty$ is $r$-tame by Proposition $\ref{minimaltameistame}$.
Since $F^*(A)$ is an $r$-connected $P_\infty$-algebra, by Theorem $\ref{existenceminimal}$
there is a Sullivan minimal $P_\infty$-algebra $\M$ together with a quasi-isomorphism $f:\M\to F^*(A)$.
The morphism $(F,f):(P_\infty,\M)\to (P,A)$ is a Sullivan $r$-minimal model of $(P,A)$.
\end{proof}

\begin{lemma}\label{quisisoAlgvar}
Let $(F,f):(P_\infty,\M)\to (P_\infty',\M')$ be a quasi-isomorphism
of Sullivan $r$-minimal algebras. Then $(F,f)$ is an isomorphism.
\end{lemma}

\begin{proof}
Since $F:P_\infty\to P_\infty'$ is a quasi-isomorphism of minimal operads, it is an isomorphism
(see Theorem 3.119. of \cite{MSS}). Therefore $F^*$ preserves Sullivan minimal algebras and
hence $f : \M \to F^*\M'$ is also an isomorphism.
\end{proof}

\begin{remark}
Note that Propodition $\ref{bijeccioalgs}$ together with Lemma $\ref{quisisoAlgvar}$
make Sullivan minimal algebras in $\Alg$,
\textit{minimal} in an abstract categorical sense (c.f. \cite{Roig3}, \cite{Roig4}, \cite{GNPR10}).
\end{remark}

\begin{theorem}
Let $A$ be an $r$-connected $P$-algebra.
Let \[(F,f):(P_\infty,\M)\to (P,A)\text{ and }(F',f'):(P'_\infty,\M')\to (P,A)\]
be two Sullivan $r$-minimal models of $(P,A)$.
Then there is an isomorphism \[(G,g):(P_\infty,\M)\to (P_\infty',\M'),\] unique up to homotopy,
such that $(F',f')\circ (G,g)\simeq (F,f)$.
\end{theorem}

\begin{proof}
By Proposition $\ref{bijeccioalgs}$ we obtain $(G,g)$, uniquely defined up to homotopy, such that $(F,f)\circ (G,g)\simeq (F',f')$.
By Lemma $\ref{quisesiso}$, $(G,g)$ is an isomorphism.
\end{proof}

Denote by $\Alg^r$ the category whose objects are pairs $(P,A)$ where $P$ is a reduced
$r$-tame operad and $A$ is an $r$-connected $P$-algebra
and by $\Ho(\Alg^r)$ the localized category with respect to quasi-isomorphisms.
Also, let $\mathbf{SMin}^r/\!\simeq $ denote the category of Sullivan $r$-minimal algebras, whose morphisms are homotopy classes of morphisms in $\Alg$.
We have:
\begin{corollary}
The category $\Alg^r$ is a Sullivan category in the sense of \cite{GNPR10}. In particular, the inclusion of minimal algebras induces an equivalence of categories
\[
\mathbf{SMin}^r/\!\simeq \stackrel{\sim}{\lra}
\Ho(\Alg^r) \ .
\]
\end{corollary}

\section{Chain operad algebras and one example}\label{SecChain}
In this section, we verify that our results are also valid
for chain operad algebras, i.e., algebras over operads in the category of
chain complexes of $\mk$-vector spaces (with homological grading).

Note that the proofs of Sections $\ref{Sectionhomotopy}$, $\ref{SecUniqueness}$ and $\ref{SecVariable}$ don't depend on
any specific behavior of the degree of differentials.
In particular, all statements and proofs admit automatic translations to the
chain setting just by replacing the word cochain by the word chain everywhere in the text,
together with the following minor changes:

\begin{enumerate}
 \item [(1)] In the definition $\ref{defKS}$ of a KS-extension
of a free $P$-algebra $A$ by a graded vector space $V'$ of degree $n$, the linear map is $d:V'\to Z_{n-1}(A)$ (instead of $Z^{n+1})$.
\item [(2)] The cone of a morphism $f:A\to B$ is in the chain setting given by
$C(f)_n=A_{n-1}\oplus B_n$ with $d(a,b)=(-da,db-f(a))$.
\item[(3)] In the definition of the algebra $\mk [t,dt]$, $dt$ has degree $-1$.
\end{enumerate}

We next revise the construction of Sullivan minimal models of Section $\ref{SecMinimals}$.
Let us remark that in the chain setting, we keep the same definition
of $r$-tame operad as in Definition $\ref{defrtame}$.
We also keep the same definition of Sullivan minimal $P$-algebra as a
colimit of a sequence of KS-extensions ordered by non-decreasing degrees.
Note that the key Lemmas $\ref{positives}$ and $\ref{comptes}$ are still valid in the chain setting, since
neither the statements nor the proofs involve any differentials.

\begin{theorem}\label{existenceminimalchain} Let $P$ be an $r$-tame operad
in chain complexes (with homological degree).
Then every $r$-connected $P$-algebra $A$ has a Sullivan $r$-minimal model
$f:\M\to A$ with  $\M_0=P(0)$ and $\M_i=0$ for all $i<r$ with $i\neq 0$.
Furthermore, if $A$ is $(r+1)$-connected and $H_*(A)$ is of finite type, then $\M$ is of finite type.
\end{theorem}

\begin{proof}
The proof is analogous to that of Theorem $\ref{existenceminimal}$ with minor changes, as done by Neisendorfer
in \cite{Ne} in the case of chain Lie algebras. Let $\M[0]=H_0(A)$ and define $f_0:\M[0]\to A$ by taking a section of the projection $Z_0(A)\twoheadrightarrow H_0(A)$.
Then $H_if_0$ is trivially an isomorphism for $i<0$ and an epimorphism for $i=0$.

Assuming we have constructed $f_{n-1} : \M[n-1]\to A$ with $\M[n-1]$ a Sullivan minimal $P$-algebra
generated in degrees $<n$ and $f_{n-1}$ a morphism such that $H_if_{n-1}$ is an isomorphism for $i<n-1$ and an epimorphism for $i=n-1$,
we build $\M[n]$ in two steps:

\begin{enumerate}
 \item [$(1_n)$] The map $f_n':\M[n]'\to A$ is obtained from $f_{n-1}:\M[n-1]\to A$ after
 killing the kernel of $f_{n-1}$ in degree $n$. This is done by successively attaching KS-extensions
 of degree $n-1$ (in the $(r+1)$-connected case, only one KS-extension is needed).
\item [$(2_n)$] The map $f_n:\M[n]\to A$ is obtained from $f_n':\M[n]'\to A$ after
killing the cokernel of $f_n'$ in dimension $n+1$. This is done by a attaching KS-extension of degree $n+1$ with trivial differential.
\end{enumerate}

Now, the resulting Sullivan $P$-algebra $\M=\bigcup_n\M[n]$ is not minimal, since KS-extensions are not ordered by degree. We next show that steps $(2_n)$ and $(1_{n+1})$ can  be permuted. Consider the sequence
\[
\cdots\to \M[n]'\to \M[n]=\M[n]'\sqcup_0\foa{U_{n+1}}\to \M[n+1]'=\M[n]\sqcup_d\foa{V_n}\to\cdots \ .
\]
Since the differential on $U_{n+1}$ is trivial, it suffices to show that
\[
d:V_n\to Z_{n-1}(\M[n-1]\setminus \foa{U_{n+1}}) \ .
\]
This is a direct consequence of the fact that $\M[n+1]_{n-1}=\M[n]'_{n-1}$, by Lemma $\ref{comptes}$.
\end{proof}

\medskip

To end the paper, we compute the minimal model of the chains of some double loop spaces $C_*\Omega^2X$ as $\Ger_\infty$-algebras. We will use the the following result, valid for either the chain or cochain setting:

\begin{proposition}\label{modelfromhomology}
Let $P$ be an operad with zero differential, $P_\infty $ its minimal model, $A$ a $P_\infty$-algebra such that $HA = P\langle V \rangle$ is a free $P$-algebra, also with zero differential. Then $P_\infty \langle V \rangle$ is the minimal model of $A$ as a $P_\infty$-algebra.
\end{proposition}

\begin{proof}
We have a $\mk$-linear lifting $s$,
\[
\xymatrix{
{}          & {}  &   {ZA} \ar[r] \ar@{->>}[d] &  {A}   \\
{V}\ar[r] \ar@{.>}[rru]^{s}  & {\foa{V}}  \ar@{=}[r]&  {HA}
}.
\]
Composed with the natural inclusion $ZA \hookrightarrow A$, $s$ will induce a morphism of $P_\infty$-algebras $\rho : P_\infty \langle V \rangle \longrightarrow A$, which is a quasi-isomorphism, since

\begin{align*}
  H(\alglibre{P_\infty}{V}) &= H \left( \bigoplus_{n\geq 0} P_\infty (n) \otimes_{\Sigma_n} V^{\otimes n} \right) 
         = \bigoplus_{n\geq 0} \left( HP_\infty (n) \otimes_{\Sigma_n} HV^{\otimes n}\right) =\\
         &= \bigoplus_{n\geq 0}\left( P(n) \otimes_{\Sigma_n} V^{\otimes n} \right) 
         = \foa{V} = HA \ .
\end{align*}
\end{proof}

\begin{example} For every connected pointed topological space $X$ its double loop space $\Omega^2 X$ has an action of the little disk operad $\D_2$, \cite{BoVo73}. Hence,
the singular chain complex $C_*\Omega^2 X$ is an algebra over the operad of the chain complex $C_*\D_2$, and the homology $H_*\Omega^2X$ is an algebra over the operad $H_*\D_2$. This is true with any coefficients, in particular over $\QQ$. From now on, we assume rational coefficients everywhere.

By the results of Cohen's thesis \cite{CLM76}, the homology $H_*\D_2$ is isomorphic to the Gerstenhaber operad: $H_*\D_2 \cong \Ger$. Therefore every homology 
$H_*\Omega^2 X$ carries a natural structure of a $\Ger$-algebra.
Furthermore, Tamarkin \cite{Tam03} showed that $C_*\D_2$ is a formal operad; that is, its minimal model in the sense of Markl is also the minimal model of its homology. This gives quasi-isomorphisms
\[
\Ger \cong H_*\D_2 \stackrel{\sim}{\longleftarrow} \Ger_\infty \stackrel{\sim}{\longrightarrow} C_*\D_2 \ .
\]
In particular, we have that every $C_*\D_2$-algebra is naturally a $\Ger_\infty$-algebra. More specifically,
$C_*\Omega^2X $ has a natural structure of a $\Ger_\infty$-algebra. We will compute the minimal model of some of these $C_*\Omega^2X$ as such.

For this, we rely on the fact that the rational homology $H_*(\Omega^2\Sigma^2 X)$ of the double loop space of the double suspension of $X$ is
free as a Gerstenhaber algebra, over the reduced homology of $X$ (see \cite{Getzler}, Section 1, cf. \cite{GeJo},
Theorem 6.1 and the generalization in \cite{SN03}, Theorem 6.5):
\[
H_*(\Omega^2\Sigma^2 X) \ \cong \ \Ger\langle \widetilde{H}_*X \rangle \ .
\]
By Proposition $\ref{modelfromhomology}$, we obtain a Sullivan minimal model as a $\Ger_\infty$-algebra
\[
\rho: \Ger_\infty \langle \widetilde{H}_*X \rangle \stackrel{\sim}{\longrightarrow} C_*(\Omega^2\Sigma^2X).
\]
For instance, for $n>2$, the minimal model of $\Omega^2S^{n+1} = \Omega^2\Sigma^2S^{n-1}$ is the free $\Ger_\infty$-algebra  $\Ger_\infty \langle e_{n-1} \rangle$ on a single generator $e_{n-1}$ in degree $n-1$. See \cite{Gi04}, Theorem 3.6, for a handy description of $ \Ger_\infty$-algebras.

This answers, we believe, a question of Getzler-Jones (\cite[Section 6]{GeJo}) about Sullivan minimal models for double loop spaces being unable to reflect the Gerstenhaber structure.
\end{example}

\subsubsection*{Acknowledgements} The first named author would like to thank Muriel Livernet and Dennis Sullivan for helpful comments.
The second named author thanks Vicen\c{c} Navarro, Pere Pascual and Francisco Guill\'{e}n for useful exchange of ideas, Neil Strickland for explaining the isomorphism describing the homology  $H_*(\Omega^2\Sigma^2 X)$ as a free Gerstenhaber algebra and Andy Tonks for providing valuable references. 
We would also like to thank Daniel Tanr\'{e} for his useful questions as well as the referees for their questions and suggestions.

\linespread{1}
\bibliographystyle{amsalpha}
\bibliography{bibliografia}

\providecommand{\bysame}{\leavevmode\hbox to3em{\hrulefill}\thinspace}
\providecommand{\MR}{\relax\ifhmode\unskip\space\fi MR }
\providecommand{\MRhref}[2]{%
  \href{http://www.ams.org/mathscinet-getitem?mr=#1}{#2}
}
\providecommand{\href}[2]{#2}
\begin{thebibliography}{GNPR10}

\bibitem[BL77]{BL}
H.~J. Baues and J.-M. Lemaire, \emph{Minimal models in homotopy theory}, Math.
  Ann. \textbf{225} (1977), no.~3, 219--242.

\bibitem[BM03]{BM03}
C.~Berger and I.~Moerdijk, \emph{Axiomatic homotopy theory for operads},
  Comment. Math. Helv. \textbf{78} (2003), no.~4, 805--831.

\bibitem[BV73]{BoVo73}
J.~M. Boardman and R.~M. Vogt, \emph{Homotopy invariant algebraic structures on
  topological spaces}, Lecture Notes in Mathematics, Vol. 347, Springer-Verlag,
  Berlin-New York, 1973.

\bibitem[Cir15]{Cirici}
J.~Cirici, \emph{Cofibrant models of diagrams: mixed {H}odge structures in
  rational homotopy}, Trans. Amer. Math. Soc. \textbf{367} (2015), no.~8,
  5935--5970.

\bibitem[CL10]{CL}
J.~Chuang and A.~Lazarev, \emph{Feynman diagrams and minimal models for
  operadic algebras}, J. Lond. Math. Soc. (2) \textbf{81} (2010), no.~2,
  317--337.

\bibitem[CLM76]{CLM76}
F.~R. Cohen, T.~J. Lada, and J.~P. May, \emph{The homology of iterated loop
  spaces}, Lecture Notes in Mathematics, Vol. 533, Springer-Verlag, Berlin-New
  York, 1976.

\bibitem[FHT01]{FHT}
Y.~F{\'e}lix, S.~Halperin, and J-C. Thomas, \emph{Rational homotopy theory},
  Graduate Texts in Mathematics, vol. 205, Springer-Verlag, New York, 2001.

\bibitem[Fre09]{Fresse}
B.~Fresse, \emph{Modules over operads and functors}, Lecture Notes in
  Mathematics, vol. 1967, Springer-Verlag, Berlin, 2009.

\bibitem[Get94]{Getzler}
E.~Getzler, \emph{Batalin-{V}ilkovisky algebras and two-dimensional topological
  field theories}, Comm. Math. Phys. \textbf{159} (1994), no.~2, 265--285.

\bibitem[Gin04]{Gi04}
G.~Ginot, \emph{Homologie et mod\`ele minimal des alg\`ebres de
  {G}erstenhaber}, Ann. Math. Blaise Pascal \textbf{11} (2004), no.~1, 95--127.

\bibitem[GJ94]{GeJo}
E.~Getzler and J.~D.~S. Jones, \emph{Operads, homotopy algebra, and iterated
  integrals for double loop spaces}, preprint hep-th/9403055 at http//arxiv.org
  (1994).

\bibitem[GM03]{GMa}
S.~Gelfand and Y.~Manin, \emph{Methods of homological algebra}, second ed.,
  Springer Monographs in Mathematics, Springer-Verlag, Berlin, 2003.

\bibitem[GM13]{GM}
P.~Griffiths and J.~Morgan, \emph{Rational homotopy theory and differential
  forms}, second ed., Progress in Mathematics, vol.~16, Springer, New York,
  2013.

\bibitem[GNPR10]{GNPR10}
F.~Guill{\'e}n, V.~Navarro, P.~Pascual, and A.~Roig, \emph{A
  {C}artan-{E}ilenberg approach to homotopical algebra}, J. Pure Appl. Algebra
  \textbf{214} (2010), no.~2, 140--164.

\bibitem[GT93]{GoTa}
A.~G{\'o}mez-Tato, \emph{Mod\`eles minimaux r\'esolubles}, J. Pure Appl.
  Algebra \textbf{85} (1993), no.~1, 43--56.

\bibitem[Hin97]{Hin97}
V.~Hinich, \emph{Homological algebra of homotopy algebras}, Comm. Algebra
  \textbf{25} (1997), no.~10, 3291--3323.

\bibitem[Hin01]{Hin01}
\bysame, \emph{Virtual operad algebras and realization of homotopy types}, J.
  Pure Appl. Algebra \textbf{159} (2001), no.~2-3, 173--185.

\bibitem[Kad80]{Kadeishvili}
T.~V. Kadei{\v{s}}vili, \emph{On the theory of homology of fiber spaces},
  Uspekhi Mat. Nauk \textbf{35} (1980), no.~3(213), 183--188, International
  Topology Conference (Moscow State Univ., Moscow, 1979).

\bibitem[KM95]{KM}
I.~K{\v{r}}{\'{\i}}{\v{z}} and J.~P. May, \emph{Operads, algebras, modules and
  motives}, Ast\'erisque (1995), no.~233, iv+145pp.

\bibitem[KP97]{KampsPorter}
K.~H. Kamps and T.~Porter, \emph{Abstract homotopy and simple homotopy theory},
  World Scientific Publishing Co. Inc., River Edge, NJ, 1997.

\bibitem[Liv98a]{Livernet1}
M.~Livernet, \emph{Homotopie rationnelle des alg\`ebres sur une op\'erade},
  1998, Th{\`e}se, Universit{\'e} Louis Pasteur (Strasbourg I), Strasbourg,
  1998.

\bibitem[Liv98b]{Livernet2}
\bysame, \emph{Rational homotopy of {L}eibniz algebras}, Manuscripta Math.
  \textbf{96} (1998), no.~3, 295--315.

\bibitem[Lod11]{Lod11}
JL. Loday, \emph{The diagonal of the {S}tasheff polytope}, Higher structures in
  geometry and physics, Progr. Math., vol. 287, Birkh\"auser/Springer, New
  York, 2011, pp.~269--292.

\bibitem[LV12]{LV}
JL. Loday and B.~Vallette, \emph{Algebraic operads}, Grundlehren der
  Mathematischen Wissenschaften [Fundamental Principles of Mathematical
  Sciences], vol. 346, Springer, Heidelberg, 2012.

\bibitem[Mil11]{Milles1}
J.~Mill{\`e}s, \emph{Andr\'e-{Q}uillen cohomology of algebras over an operad},
  Adv. Math. \textbf{226} (2011), no.~6, 5120--5164.

\bibitem[Mil12]{Milles2}
\bysame, \emph{The {K}oszul complex is the cotangent complex}, Int. Math. Res.
  Not. IMRN (2012), no.~3, 607--650.

\bibitem[MS06]{MS06}
M.~Markl and S.~Shnider, \emph{Associahedra, cellular {$W$}-construction and
  products of {$A_\infty$}-algebras}, Trans. Amer. Math. Soc. \textbf{358}
  (2006), no.~6, 2353--2372.

\bibitem[MSS02]{MSS}
M.~Markl, S.~Shnider, and J.~Stasheff, \emph{Operads in algebra, topology and
  physics}, Mathematical Surveys and Monographs, vol.~96, American Mathematical
  Society, Providence, RI, 2002.

\bibitem[Nei78]{Ne}
J.~Neisendorfer, \emph{Lie algebras, coalgebras and rational homotopy theory
  for nilpotent spaces}, Pacific J. Math. \textbf{74} (1978), no.~2, 429--460.

\bibitem[Qui69]{Q2}
D.~Quillen, \emph{Rational homotopy theory}, Ann. of Math. (2) \textbf{90}
  (1969), 205--295.

\bibitem[Roi93]{Roig4}
Agust\'{\i} Roig, \emph{Minimal resolutions and other minimal models}, Publ.
  Mat. \textbf{37} (1993), no.~2, 285--303.

\bibitem[Roi94a]{Roig1}
\bysame, \emph{Model category structures in bifibred categories}, J. Pure Appl.
  Algebra \textbf{95} (1994), no.~2, 203--223.

\bibitem[Roi94b]{Roig3}
\bysame, \emph{Mod\`eles minimaux et foncteurs d\'eriv\'es}, J. Pure Appl.
  Algebra \textbf{91} (1994), no.~1-3, 231--254.

\bibitem[Sta12]{Stan12}
A.E. Stanculescu, \emph{Bifibrations and weak factorisation systems}, Appl.
  Categ. Structures (2012), no.~20, 19--30.

\bibitem[SU04]{SU04}
S.~Saneblidze and R.~Umble, \emph{Diagonals on the permutahedra, multiplihedra
  and associahedra}, Homology Homotopy Appl. \textbf{6} (2004), no.~1,
  363--411.

\bibitem[Sul77]{Su}
D.~Sullivan, \emph{Infinitesimal computations in topology}, Inst. Hautes
  \'Etudes Sci. Publ. Math. (1977), no.~47, 269--331 (1978).

\bibitem[Sul09]{Sumaster}
\bysame, \emph{Homotopy theory of the master equation package applied to
  algebra and geometry: a sketch of two interlocking programs}, Algebraic
  topology---old and new, Banach Center Publ., vol.~85, Polish Acad. Sci. Inst.
  Math., Warsaw, 2009, pp.~297--305.

\bibitem[SW03]{SN03}
P.~Salvatore and N.~Wahl, \emph{Framed discs operads and {B}atalin-{V}ilkovisky
  algebras}, Q. J. Math. \textbf{54} (2003), no.~2, 213--231.

\bibitem[Tam03]{Tam03}
D.~E. Tamarkin, \emph{Formality of chain operad of little discs}, Lett. Math.
  Phys. \textbf{66} (2003), no.~1-2, 65--72.

\end{thebibliography}
\mbox{}\\
\linespread{1.2}

\end{document}